\documentclass[a4paper,11pt]{article}

\usepackage[all,ps,arc]{xy}
\usepackage{amsmath,amssymb,latexsym}
\usepackage{amsfonts}
\usepackage{amsthm}

\newtheorem{Theorem}{Theorem}[section]

\newtheorem{Proposition}[Theorem]{Proposition}
\newtheorem{Definition}[Theorem]{Definition}

\newtheorem{Example}[Theorem]{Example}

\newcommand{\bZ}{\ensuremath{\mathbb Z}}

\newcommand{\cC}{\ensuremath{\mathcal C}}
\newcommand{\cE}{\ensuremath{\mathcal E}}
\newcommand{\cB}{\ensuremath{\mathcal B}}

\DeclareMathOperator{\Ker}{Ker}

\newcommand\NARng{\ensuremath{\mathsf{NARng}}}
\newcommand\Rng{\ensuremath{\mathsf{Rng}}}
\newcommand\Pt{\ensuremath{\mathsf{Pt}}}
\newcommand\Act{\ensuremath{\mathsf{Act}}}

\newcommand\Gp{\ensuremath{\mathsf{Gp}}}

\newcommand\UA{\ensuremath{\mathrm{(UA)}}}
\newcommand\SH{\ensuremath{\mathrm{(SH)}}}

\newcommand\Des{\ensuremath{\mathrm{Des}}}
\newcommand\Aut{\ensuremath{\mathrm{Aut}}}

\newcommand{\diot}{\ensuremath{\divideontimes}}

\newdir{|>}{{}*!/8pt/@{|}*!/3.5pt/:(1,-.2)@^{>}*!/3.5pt/:(1,+.2)@_{>}}
\newdir{ >}{{}*!/-10pt/@{>}}
\newdir{ |>}{!/10pt/{}*!/4.5pt/@{|}*:(1,-.2)@^{>}*:(1,+.2)@_{>}}

\begin{document}

\title{Descent cospans for the fibration of points}
\author{A.~S.~Cigoli}
%
\maketitle

\begin{abstract}
Draft version of a paper concerning an interpretation of the conditon \UA~introduced in \cite{CMM17} in terms of descent with respect to the fibrations of points.
\end{abstract}


\section{Introduction}

In \cite{CMM17}, the authors introduced a condition, called \UA, that a 
semi-abelian category may or may not satisfy. Briefly, one requires that, whenever it exists, the action of a join $A\vee B$ on some object $X$ is uniquely determined by its restrictions to $A$ and $B$ respectively. The main purpose of this paper is to show that this condition can be interpreted as part of a descent property of extremal epimorphic cospans. We will prove that in a \emph{category of interest} in the sense of \cite{Orzech}  extremal epimorphic cospans are of descent with respect to the \emph{fibration of points}.

\section{The descent category of a cospan}

In this section, \cB~will be a category with pullbacks and $P\colon\cE\to\cB$ a cloven fibration, i.e.~a fibration equipped with a cleavage. We denote by $\cE_B$ the fibre of $P$ over an object $B$ of \cB, i.e.~$\cE_B=P^{-1}(B)$. For a morphism $f\colon A\to B$ in \cB, we denote by $f^*\colon\cE_B\to\cE_A$ the corresponding change-of-base functor.

Now, for a cospan
\[
\xymatrix{
A \ar[r]^{f} & B & C \ar[l]_{g}
}
\]
in the base category \cB, we would like to investigate the following question: is it possible to reconstruct the fibre $\cE_B=P^{-1}(B)$ starting from the fibres $\cE_A$ and $\cE_C$, and additional ``descent data''? This is a typical question in descent theory, and in order to formulate it properly, we need to define a suitable \emph{descent category} $\Des_{P}(f,g)$ relative to the fibration $P$ and the cospan $(f,g)$, together with a functor
\[
\Phi\colon\cE_B \longrightarrow \Des_{P}(f,g).
\]
Then, we can give the following 

\begin{Definition}
A cospan $(f,g)$ as above is of:
\begin{itemize}
 \item $P$-\emph{descent} if $\Phi$ is fully faithful;
 \item \emph{effective $P$-descent} if $\Phi$ is an equivalence of categories.
\end{itemize}
\end{Definition}

The rest of this section is devoted to a careful definition of the category $\Des_P(f,g)$ and of the functor $\Phi$. Of course, this is well-known material, that we recall here for the reader's convenience. A self-contained account on descent morphisms with respect to a cloven fibration can be found in \cite{FacDesII}. For the definition of descent category with respect to a coterminal family, one may also consult \cite[B1.5]{Elephant}.

Let us then go back to the cospan
\[
\xymatrix{
A \ar[r]^{f} & B & C \ar[l]_{g}
}
\]
in \cB. The three pullbacks
\begin{equation} \label{diag:3pb}
\xymatrix@!=4ex{
& A\times_B A \ar[dl]_{\pi_1} \ar[dr]^{\pi_2} & & A\times_B C \ar[dl]_{\pi_1} \ar[dr]^{\pi_2} & & C\times_B C \ar[dl]_{\pi_1} \ar[dr]^{\pi_2} \\
A \ar[dr]^{f} & & A \ar[dl]_{f} \ar[dr]^{f} & & C \ar[dl]_{g} \ar[dr]^{g} & & C \ar[dl]_{g} \\
& B & & B & & B
}
\end{equation}
induce, by pseudo-functoriality, three natural isomorphisms
\[
\alpha\colon \pi_1^*f^* \to \pi_2^*f^*,\qquad
\beta\colon \pi_1^*f^* \to \pi_2^*g^*,\qquad
\gamma\colon \pi_1^*g^* \to \pi_2^*g^*.
\]
Moreover, the diagonal morphisms
\[
\delta_A=\langle 1,1\rangle\colon A \to A\times_B A,\qquad\delta_C=\langle 1,1\rangle\colon C \to C\times_B C 
\]
yield the commutativity of the triangles:
\[
\xymatrix@=4ex{
\delta_A^*\pi_1^*f^* \ar[dr]_{\sim} \ar[rr]^{\delta_A^*\alpha} & & \delta_A^*\pi_2^*f^* \ar[dl]^{\sim} \\
& f^*
}
\qquad
\xymatrix@=4ex{
\delta_C^*\pi_1^*g^* \ar[dr]_{\sim} \ar[rr]^{\delta_C^*\gamma} & & \delta_C^*\pi_2^*g^* \ar[dl]^{\sim} \\
& g^*
}
\]
where the unnamed isomorphism are those induced, by pseudo-functoriality, by the equalities $\pi_1\delta_A=1_A$ and $\pi_2\delta_A=1_A$ (and similarly for $C$). Further conditions arise from higher order pullbacks. Namely, coherence of the change-of-base along all possible (backward) paths from $B$ to $A\times_BA\times_B A$ in the triple pullback
\begin{equation} \label{diag:cubeA}
\xymatrix@!=3ex{
& A \times_B A \times_B A \ar[rr]^-{\pi_{23}} \ar[dd]|{\hole}^(.7){\pi_{13}} \ar[ld]_{\pi_{12}} & & A \times_B A \ar[ld]_{\pi_{1}} \ar[dd]^{\pi_{2}} \\
A \times_B A \ar[rr]^(.7){\pi_{2}} \ar[dd]_{\pi_{1}} & & A \ar[dd]^(.3){f} \\
& A \times_B A \ar[rr]|{\hole}_(.7){\pi_{2}} \ar[dl]^{\pi_{1}} & & A \ar[dl]^{f} \\
A \ar[rr]^(.6){f} & & B
}
\end{equation}
makes the diagram
\[
\xymatrix@=4ex{
& \pi_{12}^*\pi_2^*f^* \ar[r]^{\sim} & \pi_{23}^*\pi_1^*f^* \ar[dr]^{\pi_{23}^*\alpha} \\
\pi_{12}^*\pi_1^*f^* \ar[ur]^{\pi_{12}^*\alpha} \ar[dr]_{\sim} & & & \pi_{23}^*\pi_2^*f^* \\
& \pi_{13}^*\pi_1^*f^* \ar[r]^{\pi_{13}^*\alpha} & \pi_{13}^*\pi_2^*f^* \ar[ur]_{\sim}
}
\]
commute, where the unnamed isomorphisms are induced by the commutativity of the upper, back and left face of the cube \eqref{diag:cubeA}. Similar conditions arise by coherence of all change-of-base in all the remaining triple pullbacks generated by $f$ and $g$:

{\small
\[
\xymatrix@!=2ex{
& C \times_B C \times_B C \ar[rr]^-{\pi_{23}} \ar[dd]^(.7){\pi_{13}} \ar[ld]_{\pi_{12}} & & C \times_B C \ar[ld]_{\pi_{1}} \ar[dd]^{\pi_{2}} \\
C \times_B C \ar[rr]^(.7){\pi_{2}} \ar[dd]_{\pi_{1}} & & C \ar[dd]^(.3){g} \\
& C \times_B C \ar[rr]_(.7){\pi_{2}} \ar[dl]^{\pi_{1}} & & C \ar[dl]^{g} \\
C \ar[rr]^(.6){g} & & B
}
\quad
\xymatrix@=3ex{
& \pi_{12}^*\pi_2^*g^* \ar[r]^{\sim} & \pi_{23}^*\pi_1^*g^* \ar[dr]^{\pi_{23}^*\gamma} \\
\pi_{12}^*\pi_1^*g^* \ar[ur]^{\pi_{12}^*\gamma} \ar[dr]_{\sim} & & & \pi_{23}^*\pi_2^*g^* \\
& \pi_{13}^*\pi_1^*g^* \ar[r]^{\pi_{13}^*\gamma} & \pi_{13}^*\pi_2^*g^* \ar[ur]_{\sim}
}
\quad
\]
\[
\xymatrix@!=2ex{
& A \times_B A \times_B C \ar[rr]^-{\pi_{23}} \ar[dd]^(.7){\pi_{13}} \ar[ld]_{\pi_{12}} & & A \times_B C \ar[ld]_{\pi_{1}} \ar[dd]^{\pi_{2}} \\
A \times_B A \ar[rr]^(.7){\pi_{2}} \ar[dd]_{\pi_{1}} & & A \ar[dd]^(.3){f} \\
& A \times_B C \ar[rr]_(.7){\pi_{2}} \ar[dl]^{\pi_{1}} & & C \ar[dl]^{g} \\
A \ar[rr]^(.6){f} & & B
}
\quad
\xymatrix@=3ex{
& \pi_{12}^*\pi_2^*f^* \ar[r]^{\sim} & \pi_{23}^*\pi_1^*f^* \ar[dr]^{\pi_{23}^*\beta} \\
\pi_{12}^*\pi_1^*f^* \ar[ur]^{\pi_{12}^*\alpha} \ar[dr]_{\sim} & & & \pi_{23}^*\pi_2^*g^* \\
& \pi_{13}^*\pi_1^*f^* \ar[r]^{\pi_{13}^*\beta} & \pi_{13}^*\pi_2^*g^* \ar[ur]_{\sim}
}
\]
\[
\xymatrix@!=2ex{
& A \times_B C \times_B C \ar[rr]^-{\pi_{23}} \ar[dd]^(.7){\pi_{13}} \ar[ld]_{\pi_{12}} & & C \times_B C \ar[ld]_{\pi_{1}} \ar[dd]^{\pi_{2}} \\
A \times_B C \ar[rr]^(.7){\pi_{2}} \ar[dd]_{\pi_{1}} & & C \ar[dd]^(.3){g} \\
& A \times_B C \ar[rr]_(.7){\pi_{2}} \ar[dl]^{\pi_{1}} & & C \ar[dl]^{g} \\
A \ar[rr]^(.6){f} & & B
}
\quad
\xymatrix@=3ex{
& \pi_{12}^*\pi_2^*g^* \ar[r]^{\sim} & \pi_{23}^*\pi_1^*g^* \ar[dr]^{\pi_{23}^*\gamma} \\
\pi_{12}^*\pi_1^*f^* \ar[ur]^{\pi_{12}^*\beta} \ar[dr]_{\sim} & & & \pi_{23}^*\pi_2^*g^* \\
& \pi_{13}^*\pi_1^*f^* \ar[r]^{\pi_{13}^*\beta} & \pi_{13}^*\pi_2^*g^* \ar[ur]_{\sim}
}
\]
}

Roughly speaking, the category of descent data for the cospan $(f,g)$, with respect to the fibration $P$, consists of those objects and morphisms in the fibres over $A$ and $C$, which behave ``as if'' they came from the fibre over $B$ via the change of base $f^*$ and $g^*$ respectively. More precisely:

\begin{Definition} \label{def:Des}
An object of $\Des_P(f,g)$ consists of a set $(D,F,a,b,c)$ where
\begin{itemize}
 \item $D$ is an object of $\cE_A$,
 \item $F$ is an object of $\cE_C$,
 \item $a\colon \pi_1^*D \to \pi_2^*D$ is an isomorphism in $\cE_{A \times_B A}$,
 \item $b\colon \pi_1^*D \to \pi_2^*F$ is an isomorphism in $\cE_{A \times_B C}$,
 \item $c\colon \pi_1^*F \to \pi_2^*F$ is an isomorphism in $\cE_{C \times_B C}$,
\end{itemize}
making the diagrams
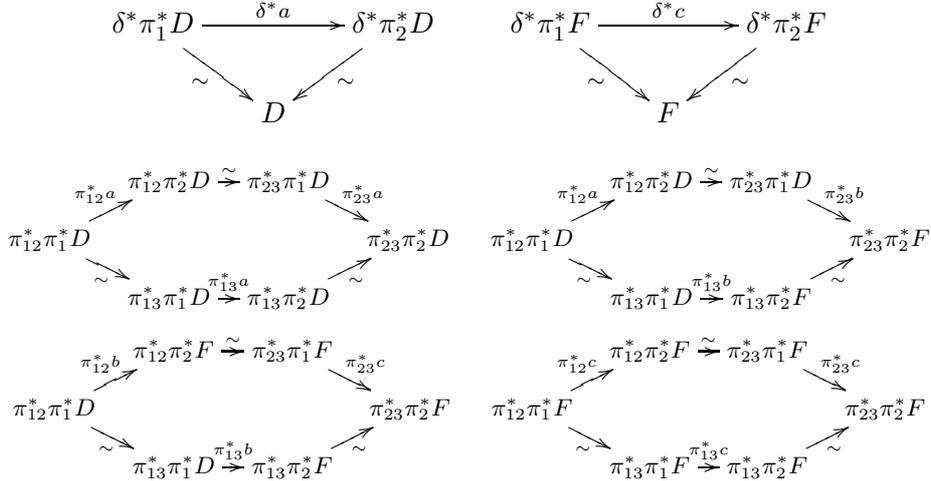
\[
\xymatrix@=4ex{
\delta^*\pi_1^*D \ar[dr]_{\sim} \ar[rr]^{\delta^*a} & & \delta^*\pi_2^*D \ar[dl]^{\sim} \\
& D
}
\qquad
\xymatrix@=4ex{
\delta^*\pi_1^*F \ar[dr]_{\sim} \ar[rr]^{\delta^*c} & & \delta^*\pi_2^*F \ar[dl]^{\sim} \\
& F
}
\]
{\footnotesize
\[
\xymatrix@=2ex{
& \pi_{12}^*\pi_2^*D \ar[r]^{\sim} & \pi_{23}^*\pi_1^*D \ar[dr]^{\pi_{23}^*a} \\
\pi_{12}^*\pi_1^*D \ar[ur]^{\pi_{12}^*a} \ar[dr]_{\sim} & & & \pi_{23}^*\pi_2^*D \\
& \pi_{13}^*\pi_1^*D \ar[r]^{\pi_{13}^*a} & \pi_{13}^*\pi_2^*D \ar[ur]_{\sim}
}
\quad
\xymatrix@=2ex{
& \pi_{12}^*\pi_2^*D \ar[r]^{\sim} & \pi_{23}^*\pi_1^*D \ar[dr]^{\pi_{23}^*b} \\
\pi_{12}^*\pi_1^*D \ar[ur]^{\pi_{12}^*a} \ar[dr]_{\sim} & & & \pi_{23}^*\pi_2^*F \\
& \pi_{13}^*\pi_1^*D \ar[r]^{\pi_{13}^*b} & \pi_{13}^*\pi_2^*F \ar[ur]_{\sim}
}
\]
\[
\xymatrix@=2ex{
& \pi_{12}^*\pi_2^*F \ar[r]^{\sim} & \pi_{23}^*\pi_1^*F \ar[dr]^{\pi_{23}^*c} \\
\pi_{12}^*\pi_1^*D \ar[ur]^{\pi_{12}^*b} \ar[dr]_{\sim} & & & \pi_{23}^*\pi_2^*F \\
& \pi_{13}^*\pi_1^*D \ar[r]^{\pi_{13}^*b} & \pi_{13}^*\pi_2^*F \ar[ur]_{\sim}
}
\quad
\xymatrix@=2ex{
& \pi_{12}^*\pi_2^*F \ar[r]^{\sim} & \pi_{23}^*\pi_1^*F \ar[dr]^{\pi_{23}^*c} \\
\pi_{12}^*\pi_1^*F \ar[ur]^{\pi_{12}^*c} \ar[dr]_{\sim} & & & \pi_{23}^*\pi_2^*F \\
& \pi_{13}^*\pi_1^*F \ar[r]^{\pi_{13}^*c} & \pi_{13}^*\pi_2^*F \ar[ur]_{\sim}
}
\]
}

\noindent commute.

An arrow $(D,F,a,b,c)\to(D',F',a',b',c')$ of $\Des_P(f,g)$ consists of a pair $(h,k)$ where
\begin{itemize}
 \item $h\colon D \to D'$ is an arrow in $\cE_{A}$,
 \item $k\colon F \to F'$ is an arrow in $\cE_{C}$,
\end{itemize}
making the diagrams
\[
\xymatrix{
\pi_1^*D \ar[r]^{\pi_1^*h} \ar[d]_{a} & \pi_1^*D' \ar[d]^{a'} \\
\pi_2^*D \ar[r]^{\pi_2^*h} & \pi_2^*D'
}
\qquad
\xymatrix{
\pi_1^*D \ar[r]^{\pi_1^*h} \ar[d]_{b} & \pi_1^*D' \ar[d]^{b'} \\
\pi_2^*D \ar[r]^{\pi_2^*k} & \pi_2^*D'
}
\qquad
\xymatrix{
\pi_1^*D \ar[r]^{\pi_1^*k} \ar[d]_{c} & \pi_1^*D' \ar[d]^{c'} \\
\pi_2^*D \ar[r]^{\pi_2^*k} & \pi_2^*D'
}
\]
commute.
\end{Definition}

The comparison functor $\Phi \colon \cE_B \to \Des_P(f,g)$ is defined on each arrow of $\cE_B$ as follows:
\[
\xymatrix@R=2.5ex{
E \ar[dd]_j & & (f^*E,g^*E,\alpha_E,\beta_E,\gamma_E) \ar[dd]^{(f^*j,g^*j)} \\
& \longmapsto \\
E' & & (f^*E',g^*E',\alpha_{E'},\beta_{E'},\gamma_{E'}),
}
\]
where the $\alpha$'s, $\beta$'s and $\gamma$'s are components of the above defined natural isomorphisms. It is immediate to check that the right hand side data fulfill the conditions of Definition \ref{def:Des}.

\section{Internal actions and the fibration of points}

In this article, we will focus on a particular fibration, namely the so-called \emph{fibration of points}. The category $\Pt(\cC)$ of points in a category \cC~is the category whose objects are pairs $(p,s)$ of arrows in \cC~such that $ps=1_B$ for some object $B$ of \cC. Morphisms are obviously defined, and if \cC~has pullbacks, the functor sending each point $(p,s)$ as above to the object $B$ (the codomain of the split epimorphism $p$) is a fibration. The reader may refer to \cite{BB} for a detailed account on this fibration and its relevance in non-abelian algebra, as well as for the notion of \emph{protomodular} category by Bourn \cite{B90}, which is a main ingredient of the notion of \emph{semi-abelian} category by Janelidze, M\'arki, and Tholen \cite{JMT}.

Change-of-base for the fibration of points are obtained by pullback
\[
\xymatrix{
A \times_B E \ar@<.5ex>[d]^{\pi_1} \ar[r]^-{\pi_2} & E \ar@<.5ex>[d]^p \\
A \ar[r]_-f \ar@<.5ex>[u]^{\langle 1,sf \rangle} & B \ar@<.5ex>[u]^s
}
\]
so that, referring to the last diagram, $f^*(p,s)=(\pi_1,\langle 1,sf \rangle)$.

If \cC\ is moreover pointed with finite coproducts, then for each object $B$ of \cC~the  change-of-base functor along the initial map, i.e.~the kernel functor
\[
\Ker_B\colon \Pt_B(\cC)\to \cC
\]
sending $(p,s)$ to the kernel of $p$, admits a left adjoint sending any object $X$ of \cC~to the point
\[
\xymatrix{
B+X \ar@<.5ex>[r]^-{[1,0]} & B \ar@<.5ex>[l]^-{\iota_B}.
}
\]
The monad on \cC~induced by this adjunction is denoted by $B\flat(-)$ and its algebras are called \emph{internal $B$-actions} (see \cite{BJK}).

These notions are of special interest in the context of semi-abelian categories, where the kernel functor above is monadic, hence yielding an equivalence of categories
\begin{equation} \label{eq:points-actions}
\Pt_B(\cC)\stackrel{\sim}{\longrightarrow}\cC^{B\flat(-)}
\end{equation}
between points over $B$ and $B$-actions. This means that \cC~has categorical \emph{semi-direct products} in the sense of \cite{BJ}. In fact, if one looks at the above equivalence when \cC~is the category of groups, one obtains the classical correspondence between group actions and semi-direct products (see \cite{BJ}).

It's worth spelling out how the functor \ref{eq:points-actions} works. It takes a point $(p,s)$ over $B$, with $\Ker_B(p,s)=X$, to the leftmost vertical arrow in the commutative diagram
\[
\xymatrix@C=6ex{
B\flat X \ar[r]^-{\ker[1,0]} \ar[d]_{\xi} & B+X \ar@<.5ex>[r]^-{[1,0]} \ar[d]_{[s,\ker(p)]} & B \ar@<.5ex>[l]^-{\iota_B} \ar@{=}[d] \\
X \ar[r]_-{\ker(p)} & E \ar@<.5ex>[r]^{p} & B. \ar@<.5ex>[l]^{s}
}
\]
One can check that $\xi$ is indeed a $B$-action.

Once we have specified a kernel of $p$, a cartesian morphism of points (i.e.~a pullback) of codomain $(p,s)$ yields a unique morphism of split extensions inducing identity on kernels:
\[
\xymatrix{
0 \ar[r] & X \ar[d]_{1_X} \ar[r]^-{\langle 0,k \rangle} & A \times_B E \ar[d]_{\pi_2} \ar@<-.7ex>[r]_-{\pi_1} & A \ar[r] \ar@<-.7ex>[l]_-{\langle 1,sf \rangle} \ar[d]^{f} & 0 \\
0 \ar[r] & X \ar[r]_-{k=\ker(p)} & E \ar@<-.7ex>[r]_{p} & B \ar@<-.7ex>[l]_{s} \ar[r] & 0,
}
\]
and this corresponds to a unique morphism
\[
\xymatrix{
A \flat X \ar[r]^{f\flat 1} \ar[dr]_{f^*\xi} & B \flat X \ar[d]^\xi \\
& X
}
\]
of actions on the object $X$. Moreover, it is easy to check, by the construction above, that isomorphic points over $B$ yield the same $B$-action on a fixed kernel.

\section{Extremal epimorphic cospans and the condition \UA}

A remarkable property of group actions is the following: if two groups $G$ and $H$ act on some group $X$, then there is a unique action of the coproduct $G+H$ on $X$ 
which restricts to the previous ones. This is a rather special feature of the category of groups, which does not hold in any semi-abelian category. However, there are interesting cases where, at least, if an action of the coproduct is defined, it is unique. This is due to the fact that those categories satisfy the condition \UA, introduced in \cite{CMM17}.
\begin{itemize}
    \item [\UA] For every extremal epimorphic cospan $\xymatrix{A \ar[r]^f & B & C \ar[l]_g}$ in \cC, and for any 4-tuple $(\xi_1,\xi_2,\xi_3,\xi_4)$ of actions on a fixed object $X$ making the diagram
\begin{equation} \label{diag:UA}
\xymatrix{
    A \flat X \ar[dr]_{\xi_1} \ar[r]^{f\flat 1} & B\flat X \ar@<-.5ex>[d]_(.4){\xi_3} \ar@<.5ex>[d]^(.4){\xi_4}
        & C \flat X \ar[dl]^{\xi_2} \ar[l]_{g\flat 1} \\
    & X
}
\end{equation}
commute, we have that $\xi_3=\xi_4$.
\end{itemize}

It is proved in \cite{CMM17} that \UA~holds in every semi-abelian category with \emph{representable} actions (see \cite{BJK}) and in every \emph{category of interest} in the sense of Orzech \cite{Orzech}. Such cases include the categories of groups, Lie algebras, rings, associative algebras and Leibniz algebras amongst others.

In fact, this condition may also be seen as part of a descent condition on the cospan $(f,g)$.

\begin{Proposition}
If a cospan $\xymatrix{A \ar[r]^f & B & C \ar[l]_g}$ in a semi-abelian category \cC~is of descent for the fibration of points, then for every commutative diagram \eqref{diag:UA}, we have that $\xi_3=\xi_4$.
\end{Proposition}

\begin{proof}
By the equivalence \eqref{eq:points-actions}, $\xi_3$ and $\xi_4$ correspond to two points $(p_3,s_3)$ and $(p_4,s_4)$, respectively, over $B$. By change-of-base along $f$, $g$ and the diagonals of the pullbacks \eqref{diag:3pb}, those two points give two objects in the descent category $\Des_{\Pt}(f,g)$. These are the images of $(p_3,s_3)$ and $(p_4,s_4)$ under the functor
\[
\Phi\colon \Pt_B(\cC) \longrightarrow \Des_\Pt(f,g).
\]
The commutative diagram \eqref{diag:UA} tells us that $f^*(p_3,s_3)\cong f^*(p_4,s_4)$ and $g^*(p_3,s_3)\cong g^*(p_4,s_4)$. These two isomorphisms clearly give an isomorphism between $\Phi(p_3,s_3)$ and $\Phi(p_4,s_4)$. Now, the cospan $(f,g)$ is of descent if and only if $\Phi$ is fully faithful, so the latter isomorphism implies that $(p_3,s_3)\cong(p_4,s_4)$, hence $\xi_3=\xi_4$.
\end{proof}

It turns out that, for categories of interest, extremal epimorphic cospans are actually of descent for the fibration of points. In order to prove this, it suffices to show that, for a cospan
\[
\xymatrix{
A \ar[r]^{f} & B & C \ar[l]_{g},
}
\]
a morphism $u\colon X \to Y$ is equivariant with respect to given $B$-actions if and only if it is equivariant with respect to the induced $A$-actions and $C$-actions respectively. This observation is made precise in the following results. For the definition of category of interest and a description of internal actions in this context, the reader may refer to \cite{CM15}.

\begin{Proposition} \label{prop:equi}
Let \cC\ be a category of interest, and
\[
\xymatrix{
A \ar[r]^{f} & B & C \ar[l]_{g}
}
\]
an extremal epimorphic cospan in \cC. Let $\xi_X$ and $\xi_Y$ be actions of $B$ on two objects $X$ and $Y$. Then a morphism $u\colon X \to Y$ in \cC\ is equivariant with respect to the actions of $B$, i.e.~the diagram
\[
\xymatrix{
B\flat X \ar[r]^-{1\flat u} \ar[d]_{\xi_X} & B\flat Y \ar[d]^{\xi_Y} \\
X \ar[r]^u & Y
}
\]
commutes, if and only if it is equivariant with respect to the induced actions of $A$ and $C$ respectively, i.e.~the two diagrams
\[
\xymatrix{
A\flat X \ar[r]^-{1\flat u} \ar[d]_{f^*\xi_X} & A\flat Y \ar[d]^{f^*\xi_Y} \\
X \ar[r]^u & Y
} \qquad
\xymatrix{
C\flat X \ar[r]^-{1\flat u} \ar[d]_{g^*\xi_X} & C\flat Y \ar[d]^{g^*\xi_Y} \\
X \ar[r]^u & Y
}
\]
commute.
\end{Proposition}

\begin{proof}
The ``only if'' part is trivial, let us prove the ``if'' part. Suppose $u$ is equivariant with respect to the induced actions of $A$ and $C$ respectively. We have to show that for each $b$ in $B$ and $x$ in $X$:
\begin{enumerate}
\item $u(b\cdot x)=b\cdot u(x)$;
\item $u(b\ast x)=b\ast u(x)$ for each $\ast$ in $\Omega_2'$.
\end{enumerate}
Without loss of generality, for the sake of simplicity, we can suppose that $A$ and $C$ are subobjects of $B$, then $(f,g)$ being extremal epimorphic means that $B=A\vee C$ and each $b$ in $B$ can be expressed as a sum of monomials of the kind
\[
a_1\ast_{1}(c_1\ast_{2}(\ldots \ast_{2n}(a_n\ast_{2n-1}c_n)))
\]
where each $\ast_i$ is an operation in $\Omega_2'$ (a binary operation other than the group operation). One can easily check that equation 1.~holds by the properties of actions. As for equation 2., it suffices to check that
\[
u((a_1\ast_1 c_1\ast_2\ldots a_n\ast_{2n-1}c_n)\ast x)=(a_1\ast_{1}c_1\ast_{2}\ldots a_n\ast_{2n-1}c_n)\ast u(x)
\]
for each such monomial (we dropped parentheses for brevity). We can proceed by induction on the length of monomials. Let us take first a monomial $a\bar{\ast}c$ of length 2, then for each $\ast$ in $\Omega_2'$ there exists a term $w$ such that
\[
(a\bar{\ast}c)\ast x=w(a\ast_1(c\star_1 x),\ldots,a\ast_r(c\star_r x),c\ast_{r+1}(a\star_{r+1}x),\ldots,c\ast_s(a\star_s x))
\]
and consequently, since $u$ is a morphism and it is equivariant with respect to the actions of $A$ and $C$ respectively, we have
\begin{align*}
u((a\bar{\ast}c)\ast x) & = u(w(a\ast_1(c\star_1 x),\ldots,a\ast_r(c\star_r x),c\ast_{r+1}(a\star_{r+1}x),\\
& \qquad\ldots,c\ast_s(a\star_s x))) = \\
& = w(u(a\ast_1(c\star_1 x)),\ldots,u(a\ast_r(c\star_r x)),u(c\ast_{r+1}(a\star_{r+1}x)),\\
& \qquad\ldots,u(c\ast_s(a\star_s x))) = \\
& = w(a\ast_1(c\star_1 u(x)),\ldots,a\ast_r(c\star_r u(x)),c\ast_{r+1}(a\star_{r+1}u(x)),\\
& \qquad\ldots,c\ast_s(a\star_s u(x))) = \\
& = (a\bar{\ast}c)\ast u(x)\,.
\end{align*}
Take now a monomial
\[
a_1\ast_{1}(c_1\ast_{2}(\ldots \ast_{2n}(a_n\ast_{2n-1}c_n)))
\]
of arbitrary length and denote $l=c_1\ast_{2}(\ldots \ast_{2n}(a_n\ast_{2n-1}c_n))$ for simplicity. Then, for each $\ast$ in $\Omega_2'$, there exists a term $w$ such that
\[
(a\ast_1l)\ast x=w(a\diot_1(l\star_1 x),\ldots,a\diot_r(l\star_r x),l\diot_{r+1}(a\star_{r+1}x),\ldots,l\diot_s(a\star_s x))\,.
\]
But since $u$ is a morphism, it is equivariant with respect to the action of $A$ and since equivariance is guaranteed on terms of length lower than $2n$, we have
\begin{align*}
u((a\ast_1l)\ast x) & = w(a\diot_1(l\star_1 u(x)),\ldots,a\diot_r(l\star_r u(x)),l\diot_{r+1}(a\star_{r+1}u(x)),\\
& \qquad\ldots,l\diot_s(a\star_s u(x))) = (a\ast_1l)\ast u(x).
\end{align*}
This completes the proof.
\end{proof}

\begin{Theorem} \label{thm:ci.des}
An extremal epimorphic cospan in a category of interest \cC~is of descent for the fibration of points.
\end{Theorem}

\begin{proof}
With the notation of Proposition \ref{prop:equi}, we have to show that the functor
\[
\Phi\colon \Pt_B(\cC) \longrightarrow \Des_\Pt(f,g)
\]
is fully faithful. First, let us prove it is full.

Consider two points $(p,s)$ and $(p',s')$ over $B$, together with chosen kernels $X$ and $Y$ for $p$ and $p'$ respectively. Let $\xi\colon B\flat X \to X$ and $\xi'\colon B\flat Y \to Y$ denote the corresponding $B$-actions. Then there are induced actions of $A$ and $C$ respectively on both $X$ and $Y$ and morphisms of actions
\[
\xymatrix{
A \flat X \ar[dr]_{f^*\xi} \ar[r]^{f\flat 1} & B\flat X \ar[d]_(.4){\xi} & C \flat X \ar[dl]^{g^*\xi} \ar[l]_{g\flat 1} \\
& X
} \qquad
\xymatrix{
A \flat Y \ar[dr]_{f^*\xi'} \ar[r]^{f\flat 1} & B\flat Y \ar[d]_(.4){\xi'} & C \flat Y \ar[dl]^{g^*\xi'} \ar[l]_{g\flat 1} \\
& Y.
}
\]
A morphism in $\Des_\Pt(f,g)$ between $\Phi(p,s)$ and $\Phi(p',s')$ gives in particular, by equivalence, two morphisms
\[
\xymatrix{
A\flat X \ar[r]^-{1\flat u} \ar[d]_{f^*\xi} & A\flat Y \ar[d]^{f^*\xi'} \\
X \ar[r]^u & Y
} \qquad
\xymatrix{
C\flat X \ar[r]^-{1\flat u} \ar[d]_{g^*\xi} & C\flat Y \ar[d]^{g^*\xi'} \\
X \ar[r]^u & Y
}
\]
of actions. By Proposition \ref{prop:equi}, a morphism
\[
\xymatrix{
B\flat X \ar[r]^-{1\flat u} \ar[d]_{\xi} & B\flat Y \ar[d]^{\xi'} \\
X \ar[r]^u & Y
}
\]
of $B$-actions is induced. By equivalence, this gives a morphism in $\Pt_B(\cC)$ between $(p,s)$ and $(p',s')$, so $\Phi$ is full.

Faithfulness then follows from the fact that \cC~is a protomodular category. For example, one can consider the commutative square of actions
\[
\xymatrix{
A\flat X \ar[rr]^-{1\flat u} \ar[dr]^{f\flat 1} \ar[dd]_{f^*\xi} & & A\flat Y \ar[dr]^{f\flat 1} \ar[dd]|{\hole}_(.3){f^*\xi'} \\
& B\flat X \ar[rr]_(.3){1\flat u} \ar[dd]^(.7){\xi} & & B\flat Y \ar[dd]^{\xi'} \\
X \ar[rr]|{\hole}^(.3)u \ar[dr]_{1_X} & & Y \ar[dr]^(.4){1_Y} \\
& X \ar[rr]^u & & Y
}
\]
which gives, by equivalence, a commutative square
\[
\xymatrix{
A\times_B E \ar[rr]^-{h} \ar[dr]^{\pi_2} \ar@<.5ex>[dd]^{\pi_1} & & A\times_B E' \ar[dr]^{\pi_2} \ar@<.5ex>[dd]|{\hole}^(.3){\pi_1} \\
& E \ar[rr]_(.3){j} \ar@<.5ex>[dd]^(.3){p} & & E' \ar@<.5ex>[dd]^{p'} \\
A \ar[rr]|{\hole}_(.7){1_A} \ar@<.5ex>[uu]^{\langle 1,sf \rangle} \ar[dr]_{f} & & A \ar[dr]^(.4){f} \ar@<.5ex>[uu]|{\hole}^(.7){\langle 1,s'f \rangle} \\
& B \ar[rr]^{1_B} \ar@<.5ex>[uu]^(.7){s} & & B \ar@<.5ex>[uu]^{s'}
}
\]
in $\Pt(\cC)$. By protomodularity, the pair $(\pi_2,s)$ of arrows with codomain $E$ is extremal epimorphic (see for example Lemma 3.1.22 in \cite{BB}). So $j$ is uniquely determined by the equations $j\pi_2=\pi_2h$ and $js=s'$, which hold for any $j\colon (p,s)\to(p',s')$ such that $f^*j=h$. This proves that $\Phi$ is faithful. 
\end{proof}

\section{Action representable context}

In this section, we draw our attention to the case where \cC~is a semi-abelian category with representable actions. We recall from \cite{BJK} that this means that the functor $\Act(-,X)$, associating with every object $X$ in \cC~the set of internal actions \emph{on} $X$, is representable. This happens, for example, in the category \Gp~of groups, where the representing object is $\Aut(X)$. In fact, it follows from results in \cite{Gray14} that not only \Gp~is action representable, but also its arrow category $\Gp^2$. The same property holds for any action representable semi-abelian category with normalizers (see Theorem 4.8 in \cite{Gray14}). It turns out that $\cC^2$ being action representable has interesting consequences (see Theorem \ref{thm:eff.des} below). Following \cite{Gray14}, we will denote by $[X]$ the object of \cC~representing actions on $X$ and by $[u]\colon [X,Y,u]\to[X]$ the object of $\cC^2$ representing actions on $u\colon X\to Y$.

We consider now cospans in \cC~which are not just extremal epimorphic, but which are colimit diagrams in the following sense: we shall suppose that the pair $(f,g)$ in the diagram
\[
\xymatrix@!=4ex{
A\times_B A \ar@<-.7ex>[dr]_{\pi_1} \ar@<.7ex>[dr]^{\pi_2} & & A\times_B C \ar[dl]_{\pi_1} \ar[dr]^{\pi_2} & & C\times_B C \ar@<-.7ex>[dl]_{\pi_1} \ar@<.7ex>[dl]^{\pi_2} \\
& A \ar[dr]_{f} & & C \ar[dl]^{g} \\
& & B
}
\]
is (together with the obvious composite arrows) the colimit cocone of the remaining part of the diagram, where all $\pi_1$'s and $\pi_2$'s are kernel pair or pullback projections. This happens, for example, when $f$ and $g$ are coproduct injections.

\begin{Theorem} \label{thm:eff.des}
In a semi-abelian category \cC, with $\cC^2$ action representable, a colimit cospan $(f,g)$ as above is of effective descent for the fibration of points. 
\end{Theorem}

\begin{proof}
By definition, we have to prove that the functor
\[
\Phi\colon \Pt_B(\cC) \to \Des_\Pt(f,g)
\]
is an equivalence of categories.

The fact that $\Phi$ is essentially surjective follows from action representability of \cC. An object of $\Des_\Pt(f,g)$ amounts to a point $(p_A,s_A)$ over $A$, a point $(p_C,s_C)$ over $C$, and chosen isomorphisms
\begin{itemize}
 \item[] $a\colon \pi_1^*(p_A,s_A) \to \pi_2^*(p_A,s_A)$ in $\Pt_{A \times_B A}(\cC)$,
 \item[] $b\colon \pi_1^*(p_A,s_A) \to \pi_2^*(p_C,s_C)$ in $\Pt_{A \times_B C}(\cC)$,
 \item[] $c\colon \pi_1^*(p_C,s_C) \to \pi_2^*(p_C,s_C)$ in $\Pt_{C \times_B C}(\cC)$.
\end{itemize}
Via the equivalece between actions and points, once a common kernel $X$ have been chosen for all points, these data give rise to a collection of actions of $A$, $C$, $A\times_B A$, $A\times_B C$, and $C\times_B C$, respectively, on $X$, which are compatible with all pullback projections. Since \cC~is action representable, one eventually gets two morphisms $\phi_A\colon A \to [X]$ and $\phi_C\colon C \to [X]$ representing the actions of $A$ and $C$ on $X$ induced by $(p_A,s_A)$ and $(p_C,s_C)$ respectively. The isomorphisms $a$, $b$, and $c$ above induce the equations
\begin{itemize}
 \item[] $\phi_A\pi_1=\phi_A\pi_2\colon A \times_B A\to B$,
 \item[] $\phi_A\pi_1=\phi_C\pi_2\colon A \times_B C\to B$,
 \item[] $\phi_C\pi_1=\phi_C\pi_2\colon C \times_B C\to B$.
\end{itemize}
So by the colimit property, there exists a unique arrow $\phi\colon B \to[X]$ such that $\phi f=\phi_A$ and $\phi g = \phi_C$. This gives an action of $B$ on $X$ which restricts to $\phi_A$ and $\phi_B$, hence, by equivalence, a point $(p,s)$ over $B$ such that $\Phi(p,s)\cong((p_A,s_A),(p_C,s_C),a,b,c)$. This proves that $\Phi$ is essentially surjective.

To prove that $\Phi$ is fully faithful one proceeds similarly, using action representability of $\cC^2$. As in the proof of Theorem \ref{thm:ci.des}, consider two points $(p,s)$ and $(p',s')$ over $B$, together with chosen kernels $X$ and $Y$ for $p$ and $p'$ respectively. A morphism in $\Des_\Pt(f,g)$ between $\Phi(p,s)$ and $\Phi(p',s')$ gives, in particular, two morphisms $h\colon f^*(p,s)\to f^*(p',s')$ in $\Pt_A(\cC)$ and $k\colon g^*(p,s)\to g^*(p',s')$ in $\Pt_C(\cC)$, with the same induced restriction $u\colon X\to Y$ between the chosen kernels. In other words, we have an object in $\Pt_{1_A}(\cC^2)$ and an object in $\Pt_{1_C}(\cC^2)$, both restricting to $u$. By action representability of $\cC^2$, these correspond to morphisms $\psi_A\colon 1_A\to[u]$ and $\psi_C\colon 1_C\to[u]$ in $\cC^2$. Let us consider now the obvious colimit diagram
\[
\xymatrix@!=4ex{
1_{A\times_B A} \ar@<-.7ex>[dr]_{(\pi_1,\pi_1)} \ar@<.7ex>[dr]^{(\pi_2,\pi_2)} & & 1_{A\times_B C} \ar[dl]^(.7){(\pi_1,\pi_1)} \ar[dr]_{(\pi_2,\pi_2)} & & 1_{C\times_B C} \ar@<-.7ex>[dl]_{(\pi_1,\pi_1)} \ar@<.7ex>[dl]^{(\pi_2,\pi_2)} \\
& 1_A \ar[dr]_{(f,f)} & & 1_C \ar[dl]^{(g,g)} \\
& & 1_B
}
\]
in $\cC^2$. Since $h$ and $k$ are part of a morphism in $\Des_\Pt(f,g)$, then $\psi_A$ and $\psi_C$ satisfy suitable equations inducing, by the colimit property, a unique arrow $\psi\colon 1_B\to [u]$ in $\cC^2$ such that $\psi(f,f)=\psi_A$ and $\psi(g,g)=\psi_C$. This gives the (unique) desired morphism $j\colon (p,s)\to(p',s')$ in $\Pt_B(\cC)$.
\end{proof}

\section{Counterexamples}

In this section we gather some counterexamples which are useful in order to distinguish different classes of categories with respect to the conditions studied above.

\begin{Example}
Extremal epimorphic cospans in the category of groups may not be of effective descent. Let us consider the symmetric group $S_3$ on three elements, presented as a semidirect product of cyclic groups $C_2$ and $C_3$ of order 2 and 3 respectively. Namely, $S_3=\langle s,\,r \mid s^2=1,\,r^3=1,\,rs=sr^2 \rangle$. The canonical inclusions
\[
\xymatrix{
C_3 \ar[r] & S_3 & C_2 \ar[l]
}
\]
form a jointly strongly epimorphic pair, since the group in the middle is generated by the other two. Now, consider the following two actions on $\bZ^3$:
\begin{align*}
0\colon C_2 \to \Aut(\bZ^3)\,, \qquad & 0(s)(x,y,z)=id_{\bZ^3}(x,y,z)=(x,y,z) \\
\rho\colon C_3 \to \Aut(\bZ^3)\,, \qquad & \rho(r)(x,y,z)=(y,z,x)
\end{align*}
We are going to show that there isn't any action of $S_3$ whose restrictions to $C_2$ and $C_3$ respectively give rise to the actions described above. Indeed, suppose such an action $\phi\colon S_3\to \Aut(X)$ exists, then
\begin{align*}
\phi(sr^2)(1,0,0) & =\phi(s)(\phi(r^2)(1,0,0)) \\
& =0(s)(\rho(r^2)(1,0,0))=id_{\bZ^3}(\rho(r)^2(1,0,0))=(0,1,0)\,,
\end{align*}
but, on the other hand $sr^2=rs$ and
\begin{align*}
\phi(rs)(1,0,0) & =\phi(r)(\phi(s)(1,0,0)) \\
& =\rho(r)(0(s)(1,0,0))=\rho(r)(1,0,0)=(0,0,1)\,.
\end{align*}
\end{Example}

\begin{Example}
Colimit cospans in the category of (not necessarily unitary) rings may not be of effective descent. Here follows an example where this property fails for coproduct injections. Let $R$ be the following ring of matrices (with the usual sum and product):
$$ R=\left\{ \left( \begin{array}{cc} x & 0 \\ y & z \end{array}\right) \bigl| x,y,z\in\bZ \right\}. $$
Let us consider the conjugation action of $R$ on itself (given by product on both sides) and the action of the ring of integers $\mathbb{Z}$ on $R$ given, for each $n$ in $\mathbb{Z}$ by
\[
n \cdot \left( \begin{array}{cc} x & 0 \\ y & z \end{array}\right) =
\left( \begin{array}{cc} x & 0 \\ y & z \end{array}\right) \cdot n =
\left( \begin{array}{cc} nx & 0 \\ 0 & nz \end{array}\right).
\]
Suppose now that an action of the coproduct $R+\mathbb{Z}$ on $R$ exists, whose restrictions give the two actions above. In that case we should have, for example, that, for each $n$ in $\mathbb{Z}$ and $r,s$ in $R$, $r(s\cdot n)=(rs)\cdot n$, but taking
\[
r=\left( \begin{array}{cc} 1 & 0 \\ 1 & 1 \end{array}\right),
\quad s=\left( \begin{array}{cc} 1 & 0 \\ 1 & 1 \end{array}\right), \quad n=1,
\]
one finds
\[
(rs)\cdot n=\left(\left( \begin{array}{cc} 1 & 0 \\ 1 & 1 \end{array}\right)
\left( \begin{array}{cc} 1 & 0 \\ 1 & 1 \end{array}\right)\right)\cdot 1 =
\left( \begin{array}{cc} 1 & 0 \\ 2 & 1 \end{array}\right)\cdot 1 =
\left( \begin{array}{cc} 1 & 0 \\ 0 & 1 \end{array}\right),
\]
while
\[
r(s\cdot n) = \left( \begin{array}{cc} 1 & 0 \\ 1 & 1 \end{array}\right)
\left(\left( \begin{array}{cc} 1 & 0 \\ 1 & 1 \end{array}\right)\cdot 1 \right) =
\left( \begin{array}{cc} 1 & 0 \\ 1 & 1 \end{array}\right)
\left( \begin{array}{cc} 1 & 0 \\ 0 & 1 \end{array}\right) =
\left( \begin{array}{cc} 1 & 0 \\ 1 & 1 \end{array}\right).
\]
So no compatible action is defined of $R+\mathbb{Z}$ on $R$.
\end{Example}

\begin{Example}
In the category \NARng~of non-associative rings extremal epimorphic cospans are not even of descent for the fibration of points. Let us consider the object given by the abelian group on three generators $A=\bZ x + \bZ y + \bZ z$, equipped with a distributive binary operation defined on generators as
\[
\begin{array}{r|ccc}
	\cdot & x & y & z \\
	\hline
	x & 0 & z & 0 \\
	y & z & 0 & 0 \\
	z & 0 & 0 & 0 \\
\end{array}
\]
This is indeed a ring. However, in \NARng, one can define an internal action $\xi\colon A\flat \bZ \to \bZ$ in the following way:
\[ (ax+by+cz)*n=n*(ax+by+cz)=cn. \]
Notice that this is not an action in \Rng, since, for example, $(z\cdot z)*1=0\neq 1=z*(z*1)$. On the other hand, we also have a trivial action $\tau$ of $A$ on \bZ, defined by:
\[ (ax+by+cz)\bullet n=n \bullet(ax+by+cz)=0. \]
If we consider now the subobjects $i\colon X=\bZ x \to A$ and $j\colon Y=\bZ y \to A$, then clearly $A= X \vee Y$. Moreover, both actions $\xi$ and $\tau$ restrict to the same actions of $X$ and $Y$ respectively on \bZ, i.e.\ the trivial actions. Namely, $i^*(\xi)=i^*(\tau)$ and $j^*(\xi)=j^*(\tau)$.

This example shows that in \NARng\ the condition \UA\ does not hold. In addition, \NARng\ being a category of distributive $\Omega_2$-groups, which form a strongly protomodular category (see \cite{MM-Peiffer}) this also proves that strong protomodularity does not imply \UA.
\end{Example}

\section{Connection with \SH}

In this section, we shall consider the following reformulation of the property \UA, which makes sense in any protomodular category \cC. 
\begin{itemize}
\item [\UA] For every extremal epimorphic cospan $\xymatrix{A \ar[r]^f & B & C \ar[l]_g}$ in \cC, and for any commutative diagram of solid arrows
\[
\xymatrix@!=5ex{
A\times_B E \ar[dr]^{\alpha}_{\sim} \ar@<.5ex>[dd]^{\pi_1} \ar[rr]^{\pi_2} & & E \ar@<.5ex>[dd]|{\hole}^(.7){p} \ar@{-->}[dr]^{\psi}_{\sim} & & C \times_B E \ar[dr]^{\gamma}_{\sim} \ar@<.5ex>[dd]|{\hole}^(.7){\pi_1} \ar[ll]_{\pi_2} \\
& A\times_B E' \ar@<.5ex>[dd]^(.3){\pi_1} \ar[rr]^(.7){\pi_2} & & E' \ar@<.5ex>[dd]^(.3){p'} & & C \times_B E' \ar@<.5ex>[dd]^{\pi_1} \ar[ll]_(.7){\pi_2} \\
A \ar@{=}[dr] \ar[rr]|{\hole}_(.7)f \ar@<.5ex>[uu]^{\langle 1,sf \rangle} & & B \ar@{=}[dr] \ar@<.5ex>[uu]|{\hole}^(.3){s} & & C \ar@{=}[dr] \ar[ll]|{\hole}^(.3)g \ar@<.5ex>[uu]|{\hole}^(.3){\langle 1,sg \rangle} \\
& A \ar[rr]^{f} \ar@<.5ex>[uu]^(.7){\langle 1,s'f \rangle} & & B \ar@<.5ex>[uu]^(.7){s'} & & C \ar@<.5ex>[uu]^{\langle 1,s'g \rangle} \ar[ll]_{g}
}
\]
where all vertical pairs are points and $\alpha$ and $\beta$ are isomorphisms, there exists a unique dashed arrow $\psi$ giving a morphisms of points between $(p,s)$ and $(p',s')$. This arrow is necessarily an isomorphism by protomodularity.
\end{itemize}

We shall now compare this condition with the condition
\begin{itemize}
 \item [\SH] Two equivalence relations on the same object of \cC~centralize each other as soon as their associated normal subobjects cooperate.
\end{itemize}
This condition makes sense in any pointed protomodular category \cC. The reader may refer to \cite{BB} for a detailed account on centralization of equivalence relations and on cooperating morphisms.

\begin{Proposition}
In any pointed protomodular category, \UA\ implies \SH.
\end{Proposition}

\begin{proof}
%
Let $R$ and $S$ be two effective equivalence relations on the same object $A$
\[
\xymatrix{
R \ar@<-1ex>[r]_{r_1} \ar@<1ex>[r]^{r_2} & A \ar[l]|{e_R} \ar[r]|{e_S} & S \ar@<-1ex>[l]_{s_1} \ar@<1ex>[l]^{s_2} 
}
\]
and let $h\colon X \to A$ and $k\colon Y\to A$ (respectively) be the corresponding kernel maps. We want to prove that, under \UA, $R$ and $S$ centralize each other as soon as $h$ and $k$ admit a cooperator $\phi$. Let us consider the following diagram, where the squares $r_2p_1=s_1q_2$ and $r_2p_2'=s_2q_2'$ are pullbacks, $v$ is the unique arrow in $P_1$ induced by $\langle 0,k\rangle\pi_2\colon X \times Y \to S$ and $\langle h,0\rangle\pi_1\colon X \times Y \to R$, and $u$ is the unique arrow in $P_2$ induced by $\langle 0,k\rangle\pi_2\colon X \times Y \to S$ and $\langle \phi,k\pi_2\rangle\colon X \times Y \to R$
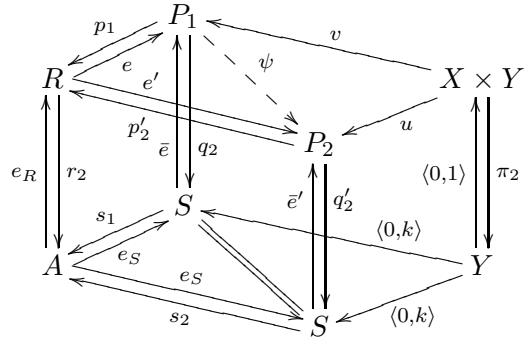
\[
\xymatrix@R=2ex@C=7ex{
& P_1 \ar@<-.5ex>[dl]_{p_1} \ar@<.5ex>[ddd]^(.7){q_2} \ar@{-->}[ddr]^{\psi} \\
R \ar@<.5ex>[ddd]^{r_2} \ar@<.5ex>[drr]^(.35){e'} \ar@<-.5ex>[ur]_{e} & & & X\times Y \ar@<.5ex>[ddd]^{\pi_2} \ar[dl]^{u} \ar[ull]_{v} \\
& & P_2 \ar@<.5ex>[ull]^(.65){p'_2} \ar@<.5ex>[ddd]^(.3){q'_2} \\
& S \ar@<-.5ex>[dl]_{s_1} \ar@{=}[ddr] \ar@<.5ex>[uuu]^(.3){\bar{e}} \\
A \ar@<-.5ex>[ur]_{e_S} \ar@<.5ex>[drr]^{e_S} \ar@<.5ex>[uuu]^{e_R} & & & Y \ar[dl]^{\langle 0,k\rangle} \ar[ull]_(.3){\langle 0,k\rangle} \ar@<.5ex>[uuu]^{\langle 0,1\rangle} \\
& & S \ar@<.5ex>[ull]^{s_2} \ar@<.5ex>[uuu]^(.7){\bar{e}'}
}
\]
Since the squares $q_2'e'=e_Sr_2$ and $q_2e=e_Sr_2$ are both pullbacks, the point $(r_2, e_R)$ is isomorphic to $e_S^*(q_2',\bar{e}')$ and $e_S^*(q_2,\bar{e})$ at the same time. On the other hand, the squares $q_2'u=\langle 0,k\rangle\pi_2$ and $q_2v=\langle 0,k\rangle\pi_2$ are also pullbacks, showing that the point $(\pi_2,\langle 0,1\rangle)$ on the right hand side is isomorphic to $\langle 0,k\rangle^*(q_2',\bar{e}')$ and $\langle 0,k\rangle^*(q_2,\bar{e})$ at the same time. By \UA, this implies that there exists an isomorphism $\psi\colon P_1 \to P_2$ representing a morphism of points between $(q_2,\bar{e})$ and $(q_2',\bar{e}')$ and such that $\psi e=e'$ and $\psi v=u$. Finally, putting $p_2=p_2'\psi$, we get a pair of discrete fibrations
\[
\xymatrix@=7ex{
P_1 \ar@<-1ex>[d]_{p_1} \ar@<1ex>[d]^{p_2} \ar@<.5ex>[r]^{q_2} & S \ar@<-1ex>[d]_{s_1} \ar@<1ex>[d]^{s_2} \ar@<.5ex>[l]^{\bar{e}} \\
R \ar@<.5ex>[r]^{r_2} \ar[u]|{e} & A. \ar[u]|{e_S} \ar@<.5ex>[l]^{e_R}
}
\]
We will prove that $R$ and $S$ centralize each other by showing that the composite $r_1p_2$ provides a connector. It was proved in \cite{BG02} that, in a Mal'tsev context, the existence of a connector for $R$ and $S$ is equivalent to the fact the they centralize each other.
The following equalities ensure that $r_1p_2$ is indeed a connector for $R$ and $S$ over $A$:
\[
\left\{
\begin{array}{l}
r_1p_2e=r_1p_2'\psi e=r_1p_2'e'=r_1 \\
r_1p_2\bar{e}=r_1p_2'\psi \bar{e}=r_1p_2'\bar{e}'=r_1e_Rs_2=s_2.
\end{array}
\right.
\]
\end{proof}

\end{document}